\documentclass[12pt]{amsart}
\usepackage{enumerate,amssymb,version,aliascnt,geometry,amsmath}
\usepackage[all]{xy}
\usepackage{hyperref}
\usepackage{datetime}
\usepackage{enumitem}
\usepackage{cleveref}
\hypersetup{
pdfauthor={Olivier Haution},
pdftitle={Detection by regular schemes in degree two},
hidelinks=true,
pdfkeywords={Steenrod operations, Chern character, Resolution of singularities, Bivariant classes, Operational Chow group, Witt indices}
}

\DeclareMathOperator{\id}{id}
\DeclareMathOperator{\Spec}{Spec}
\DeclareMathOperator{\Proj}{Proj}
\DeclareMathOperator{\im}{im}
\DeclareMathOperator{\CH}{CH}
\DeclareMathOperator{\A}{A}
\DeclareMathOperator{\B}{B}
\DeclareMathOperator{\Pic}{Pic}
\DeclareMathOperator{\ch}{ch}
\DeclareMathOperator{\Sq}{Sq_1}
\DeclareMathOperator{\Squ}{Sq^2}
\DeclareMathOperator{\Sing}{Sing}

\DeclareMathOperator{\trdeg}{tr.deg.}
\DeclareMathOperator{\codim}{codim}

\newcommand{\iu}{\mathfrak{i}_1}
\newcommand{\Oc}{\mathcal{O}}
\newcommand{\Sc}{\mathcal{S}}
\newcommand{\Cc}{\mathcal{C}}
\newcommand{\Dc}{\mathcal{D}}
\newcommand{\Ec}{\mathcal{E}}

\newcommand{\Sch}{\mathsf{Sch}_S}
\newcommand{\Scht}{\mathsf{Sch}_T}
\newcommand{\Hom}{\mathsf{Hom}}
\newcommand{\Ab}{\mathsf{Ab}}
\newcommand{\coshv}{\mathsf{coShv}}

\newcommand{\Zz}{\mathbb{Z}}
\newcommand{\Qq}{\mathbb{Q}}
\newcommand{\Pp}{\mathbb{P}}
\newcommand{\Aa}{\mathbb{A}}

\newcommand{\K}{K}
\newcommand{\Tan}{T}
\newcommand{\Ko}[2]{K_{#1}(\Oc_{#2})}
\newcommand{\HK}[3]{H^{#1}_{#2}({#3},\Ko{#1}{#3})}

\newtheorem{theorem}{Theorem}[section]

\newaliascnt{proposition}{theorem}
\newtheorem{proposition}[proposition]{Proposition}
\aliascntresetthe{proposition}

\newaliascnt{lemma}{theorem}
\newtheorem{lemma}[lemma]{Lemma}
\aliascntresetthe{lemma}

\newaliascnt{corollary}{theorem}
\newtheorem{corollary}[corollary]{Corollary}
\aliascntresetthe{corollary}

\newaliascnt{conjecture}{theorem}
\newtheorem{conjecture}[conjecture]{Conjecture}
\aliascntresetthe{conjecture}

\theoremstyle{definition}

\newaliascnt{remark}{theorem}
\newtheorem{remark}[remark]{Remark}
\aliascntresetthe{remark}

\newaliascnt{example}{theorem}
\newtheorem{example}[example]{Example}
\aliascntresetthe{example}

\newaliascnt{definition}{theorem}
\newtheorem{definition}[definition]{Definition}
\aliascntresetthe{definition}

\begin{document}
\begin{abstract}
Using Lipman's results on resolution of two-dimensional singularities, we provide a form of resolution of singularities in codimension two for reduced quasi-excellent schemes. We deduce that operations of degree less than two on algebraic cycles are characterised by their values on classes of regular schemes. We provide several applications of this ``detection principle'', when the base is an arbitrary regular excellent scheme: integrality of the Chern character in codimension less than three, existence of weak forms of the second and third Steenrod squares, Adem relation for the first Steenrod square, commutativity and Poincar\'e duality for bivariant Chow groups in small degrees. We also provide an application to the possible values of the Witt indices of non-degenerate quadratic forms in characteristic two.
\end{abstract}
\author{Olivier Haution}
\title{Detection by regular schemes in degree two}
\email{olivier.haution at gmail.com}
\address{Mathematisches Institut, Ludwig-Maximilians-Universit\"at M\"unchen, Theresienstr.\ 39, D-80333 M\"unchen, Germany}



\maketitle

\section*{Introduction}
Resolution of singularities, which at the moment is only available for $\Qq$-schemes, is a fundamental tool for many questions related to algebraic cycles, for example for the study of operations on Chow groups (see e.g.\ \cite{reduced}). When considering operations of degree $\leq n$, it is often sufficient to resolve singularities of codimension $\leq n$, that is, given a scheme $X$, to find a proper birational morphism $f\colon X' \to X$ such that the singular locus of $X'$ has codimension $>n$. For $n=1$ this remark has been exploited in \cite{firstsq}, together with the observation that we can use the normalisation morphism as a resolution of singularities of codimension one. In the present paper, we observe that it is enough to require that $f$ desingularise $X$ up to codimension $n$ (in the terminology of \cite{Temkin-desing,Temkin-duke}), which means that the singular locus of $X'$ is mapped to codimension $>n$ in $X$. We prove in \autoref{th:res} that desingularisation up to codimension two is possible: any reduced quasi-excellent scheme $X$ admits a proper birational morphism $f \colon \widetilde{X} \to X$ mapping the singular locus of $\widetilde{X}$ to a subset of codimension at least three in $X$.

We then provide some applications to algebraic cycles on schemes over a regular base, using \autoref{th:res} as a replacement for resolution of singularities.\\

It is worth noting that something close to desingularisation up to codimension
$n$ was used already by Hironaka in \cite{Hir-64} under the name ``localization of resolution data'', see \cite[Chapter IV, section 1]{Hir-64} (e.g.\ Propositions 1 and 2 in loc.cit.). However, both in \cite{Hir-64} or \cite{Temkin-desing,Temkin-duke} desingularisation up to codimension $n$ is only used as a subroutine of a global desingularisation method or algorithm; one of the purposes of this paper is to illustrate how this procedure can be used outside of the realm of desingularisation theory.\\

The paper is organised as follows. In \autoref{sect:conventions}, we fix the conventions used in the text.

In \autoref{sect:resolv}, we prove the main result, \autoref{th:res} mentioned above. The strategy for the proof is the following. We may assume that $X$ is normal. Then there are only finitely many singular points $s_1,\dots,s_n$ of codimension $\leq 2$ in $X$. Using resolution of singularities in dimension two, we can resolve each of the local schemes $Y_i=\Spec \Oc_{X,s_i}$. The next step is to extend these local resolutions to a proper birational morphism $f \colon \widetilde{X} \to X$ inducing an isomorphism outside of the closure of $\{s_1,\dots,s_n\}$. This is possible since a resolution of singularities of $Y_i$ may be achieved by the blow-up along a \emph{zero-dimensional} closed subscheme (we refer to a paper of Lipman for this result\footnote{as pointed out by the referee, the minimal desingularisation input that we need is the following: if $X$ is a quasi-excellent two-dimensional normal local scheme, then there is a regular scheme $X'$, and a proper birational morphism $X' \to X$ which induces an isomorphism over the complement of the closed point in $X$}). Any singular point $x$ of $\widetilde{X}$ lies over a singular point $f(x)$ of $X$; on the other hand if $f(x)$ has codimension two in $X$, it is one the points $s_1,\dots,s_n$, and it follows from the construction that $\widetilde{X}$ must be regular at $x$.

In \autoref{sect:procedure}, we translate the main result into a technical statement (\autoref{lemm:test1}) adapted to the situations considered in each of the subsequent sections. It expresses the idea that the vanishing of an operation on algebraic cycles, which lowers the dimension by $\leq 2$, is detected by its action on the fundamental classes of regular schemes.

In \autoref{sect:integrality}, we deduce an integrality property for the Chern character in codimension two, and then show that it actually implies the same property in codimension three. We construct two operations on Chow groups modulo two (and two-torsion for the target), which are related to the second and third Steenrod squares. As an application, we deduce a result contained in Hoffmann's conjecture, concerning Witt indices of quadratic forms; this is a new statement when the base field has characteristic two.

In \autoref{sect:adem} we extend, by descent, the definition of the first Steenrod square $\Sq$, given in \cite{duality} for quasi-projective varieties, to arbitrary schemes of finite type over a field. We prove the relation $\Sq \circ \Sq=0$, which was apparently out of reach of the techniques used in \cite{duality}. The argument, as well as the construction of $\Sq$, do not depend on the characteristic of the base field; in characteristic different from two, a different construction of the operation $\Sq$ has been given in \cite{Bro-St-03,Vo-03}, and the relation $\Sq \circ \Sq=0$ was known, being one of the Adem relations.

In \autoref{sect:bivariant}, we discuss consequences for the bivariant Chow group. We prove an absolute form of Poincar\'e duality, which implies that the operational and usual Chow groups of a regular excellent scheme coincide in degrees $\leq 2$. We also discuss commutativity of bivariant classes of small degree; in particular we prove that any class of degree $\leq 2$ commuting with proper push-forwards, and pull-backs along smooth morphisms and regular closed embeddings automatically commutes with flat pull-backs. Here resolution singularities is combined with another, a priori independent, feature of codimension $\leq 2$: bivariant classes of degree $\leq 2$ can be expressed using Chern classes with supports. In higher degrees, denominators appear, and one needs other sources of bivariant classes, such as $K$-cohomology. Finally we explain how these results can be extended to classes of arbitrary degrees in characteristic zero, where both resolution of singularities and Gersten's conjecture are available.

\paragraph{\textbf{Acknowledgements.}} I am grateful to the referee for his remarks, which helped improve the exposition.

\section{Terminology and conventions}
\label{sect:conventions}
\subsection{Schemes and morphisms} All schemes are noetherian and separated. When $S$ is a scheme, we denote by $\Sch$ the category schemes of finite type over $S$ and proper $S$-morphisms.

We say that a morphism $Y \to X$ is, or induces, an isomorphism outside of a closed subset $Z$ of $X$ if the base change $(X-Z)\times_X Y \to X-Z$ is an isomorphism.

\subsection{Regularity} We say that a scheme $X$ is \emph{regular at a point $x$} when the local ring $\Oc_{X,x}$ is regular. Otherwise we say that $x$ is a \emph{singular point} of $X$. The \emph{singular locus} $\Sing X$ is the set of such points. A scheme is \emph{regular} if it is regular at each of its points.

\subsection{Normalisation} When $X$ is an integral scheme, its \emph{normalisation} $X' \to X$ is the affine morphism given locally by the integral closure of its coordinate ring in its function field. More generally, when $X$ is an arbitrary scheme, its normalisation is the composite $\coprod_i X_i' \to\coprod_i X_i\to X$, where $X_i$ are the irreducible components of $X$, and $X_i' \to X_i$ their normalisations. 

\subsection{Excellence} A scheme $S$ is called \emph{quasi-excellent} if
\begin{itemize}[label=---]
\item any integral scheme finite over $S$ admits an open dense regular subscheme,

\item and for any closed point $x$ of $X$, the completion morphism of $\Spec \Oc_{X,x}$ has geometrically regular fibres.
\end{itemize}
If $S$ is quasi-excellent, then any scheme of finite type over $S$ is quasi-excellent, and so is the localisation of $S$ at any of its points \cite[\S2, Th\'eor\`eme~5.1]{Gabber-book}. The spectrum of a field, and of $\Zz$, is quasi-excellent, so that the class of quasi-excellent schemes contains all varieties and arithmetic schemes. The normalisation of a quasi-excellent scheme is finite \cite[(7.6.1)]{ega-4}; the singular locus of a quasi-excellent scheme is closed \cite[(6.12.3)]{ega-4}. For a scheme of finite type over a regular scheme, we will use the terminology \emph{excellent} instead of quasi-excellent; this is compatible with \cite[(7.8.2)]{ega-4}.

\subsection{Projectivity} A morphism $Y \to X$ is \emph{projective} if there is a quasi-coherent sheaf of graded $\Oc_X$-algebras $\Sc_{\bullet}$, with $\Sc_1$ a coherent $\Oc_X$-module generating $\Sc_{\bullet}$ as an $\Oc_X$-algebra, and an isomorphism over $X$ between $Y$ and $\Proj_X \Sc_{\bullet}$. A morphism is \emph{quasi-projective} if it decomposes as an open immersion followed by a projective morphism.

\subsection{Envelopes}
\label{def:envelope}
A proper morphism $f \colon Y \to X$ is an \emph{envelope} if for any integral closed subscheme $Z$ of $X$, there is a closed subscheme $W$ of $Y$ such that $f$ induces a birational morphism $W \to Z$. It is equivalent to require that $Y(K) \to X(K)$ be surjective for all fields $K$. When $X$ is of finite type over $S$, a \emph{Chow envelope} over $S$ is an envelope $Y\to X$ with $Y$ quasi-projective over $S$. By Chow's Lemma \cite[(5.6.1)]{ega-2} and noetherian induction, such an envelope always exists. 

\subsection{Dimension}
\label{conv:reldim}
The \emph{codimension} $\codim_X(x)$ of a point $x$ in a scheme $X$ is the dimension of the local ring $\Oc_{X,x}$. When $Y$ is a closed subset of $X$, its codimension $\codim_X(Y)$ is the infimum of the codimensions in $X$ of the points of $Y$. When $X$ is an integral scheme of finite type over a scheme $S$, the \emph{dimension of $X$ over $S$} is defined as
\[
\dim_S X = \trdeg(\kappa(X)/\kappa(\overline{X})) - \codim_S(\overline{X}),
\]
where $\overline{X}$ is the scheme theoretic image of $X \to S$, and $\kappa(X)$, resp.\ $\kappa(\overline{X})$, the function field of $X$, resp.\ $\overline{X}$. When $Z$ is an integral closed subscheme of $X$, we have, using the dimension formula \cite[(5.6.5.1)]{ega-4},
\begin{align*}
\dim_S Z 
&= \trdeg(\kappa(Z)/\kappa(\overline{Z})) - \codim_S(\overline{Z}) \\ 
&\leq \trdeg(\kappa(Z)/\kappa(\overline{Z})) - \codim_{\overline{X}}(\overline{Z})-\codim_S(\overline{X}) \\
 &\leq \trdeg(\kappa(X)/\kappa(\overline{X})) - \codim_X(Z) - \codim_S(\overline{X})\\
 &= \dim_S X - \codim_X(Z).
\end{align*}
In particular $\dim_S Z \leq \dim_S X$, so that when $Y$ is a scheme of finite type over $S$, we may define $\dim_S Y$ as the supremum of the integers $\dim_S Z$ for $Z$ an integral closed subscheme of $Y$. 

Let $X$ be an integral scheme of finite type over $S$, and $Y$ a closed subscheme of $X$. Since for any integral closed subscheme $Z$ of $Y$, we have $\dim_S Z \leq \dim_S X - \codim_X(Z) \leq \dim_S X - \codim_X(Y)$, it follows that
\begin{equation}
\label{eq:dimS}
\dim_S Y \leq \dim_S X - \codim_X(Y).
\end{equation}

\subsection{Abelian groups} 
\label{conv:abelian}
We denote by $\Ab$ the category of abelian groups. Let $p$ be a prime number. We denote by $\Zz_{(p)}$ the subgroup of $\Qq$ consisting of those fractions whose denominator is prime to $p$. When $A$ is an abelian group and $a \in A \otimes_\Zz \Qq$, we let $v_p(a)$ be the largest integer $k$ such that $p^{-k}\cdot a$ belongs to the image of $A \otimes_\Zz \Zz_{(p)} \to A\otimes_\Zz \Qq$ (we set $v_p(0)=\infty$). When $t$ is an integer, the group $A \otimes_\Zz (\Qq/(p^{-t}\cdot \Zz_{(p)}))$ may be identified with the quotient of the group $A\otimes_\Zz \Qq$ by the subgroup of elements $a$ satisfying $v_p(a) \geq -t$.

\subsection{Chow groups, Grothendieck groups}
\label{conv:chow}
Let $S$ be a regular scheme. The Grothendieck group of coherent sheaves gives a functor $\K_0'\colon \Sch \to \Ab$. For an object $X$ of $\Sch$, the subgroup $\K_0'(X)_{(d)} \subset \K_0'(X)$ is generated by the elements $[\Oc_Z]$, where $Z$ is a closed subscheme of $X$ with $\dim_S Z \leq d$ (see \ref{conv:reldim}). The Chow group of cycles of dimension $j$ over $S$ is a functor $\CH_j\colon \Sch \to \Ab$, see \cite[\S20]{Ful-In-98}. We have the $j$-th homological Chern character, a natural transformation (this is the $j$-th component of the morphism $\tau$ of \cite[Theorem~18.3, \S20]{Ful-In-98})
\[
\ch_j \colon \K_0'(-) \to \CH_j(-) \otimes_\Zz \Qq.
\]

\section{Resolving singularities in codimension two}
\label{sect:resolv}

\begin{theorem}
\label{th:res}
Let $X$ be a reduced quasi-excellent scheme. Then there is a reduced scheme $\widetilde{X}$ and a projective birational morphism $f \colon \widetilde{X} \to X$ with the following properties.
\begin{enumerate}[label=(\roman{*}), ref=(\roman{*})]
\item \label{cond:reg} The scheme $\widetilde{X}$ is regular at any of his points $x$ such that $f(x)$ has  codimension $\leq 2$ in $X$.

\item \label{cond:isom} The morphism $f$ is an isomorphism outside of the singular locus of $X$.
\end{enumerate}
\end{theorem}
\begin{proof}
We first assume that $X$ is normal. Then $\Sing X$ contains no point of codimension $\leq 1$ in $X$. Let $\Sigma$ be the set of points of $\Sing X$ which have codimension $2$ in $X$. We view $\Sigma$ and $\Sing X$ as topological subspaces of $X$. Then $\Sigma$ consists of generic points of $\Sing X$. Since $X$ is quasi-excellent, $\Sing X$ is the support of a noetherian scheme, and therefore the space $\Sigma$ is finite and discrete. Thus we can write $\Sigma=\{s_1,\cdots,s_n\}$, with $s_i$ not belonging to the closure of $\{s_j|j\neq i\}$ in $X$.

For every $i \in \{1,\dots,n\}$, the scheme $Y_i=\Spec \Oc_{X,s_i}$ is a singular normal quasi-excellent scheme of dimension two. By \cite[Theorem p.151, and C p.155]{Lipman-desingularization}, there is a zero-dimensional closed subscheme $Z_i$ of $Y_i$ such that the blow-up $g_i \colon \widetilde{Y_i} \to Y_i$ of $Y_i$ along $Z_i$ is a desingularisation, i.e.\ $\widetilde{Y_i}$ is regular. 

We let $Z=\coprod_{i=1}^n Z_i$ and $Y=\coprod_{i=1}^n Y_i$, and claim that $z \colon Z \to Z \times_X Y$ is an isomorphism. Indeed, since the open subschemes $Z_j \times_X Y_i$ for $i,j = 1,\dots,n$ cover $Z \times_X Y$, it suffices to see that the base change $z_{i,j}\colon Z_j \times_Y Y_i \to Z_j \times_X Y_i$ of $z$ is an isomorphism. If $j \neq i$,  no point of $Y_i$ is mapped to $s_j$, so that $Z_i \times_X Y_j = \emptyset$ and $z_{i,j}$ is an isomorphism. The morphism $Y_i \to X$ is a monomorphism by \cite[I (2.4.2)]{ega-1}, and so is $Y_i \to Y$. Therefore $d_i \colon Z_i \to Z_i \times_Y Y_i$ and $z_{i,i} \circ d_i \colon Z_i \to Z_i \times_X Y_i$ are isomorphisms, and we deduce that $z_{i,i}$ is an isomorphism.

Let $W$ be the scheme theoretic image of $Z\to X$ and $f\colon \widetilde{X} \to X$ be the blow-up of $X$ along $W$. The morphism $Y \to X$ is flat, and by compatibility of scheme theoretic images with flat base change \cite[(2.3.2)]{ega-4}, the scheme theoretic image of the morphism $Z= Y \times_X Z \to Y$ is $Y \times_X W$. Since this morphism is a closed embedding, it follows that $Z \to Y \times_X W$ is an isomorphism. By compatibility of blow-ups with flat base change, we see that the base change of $f$ along $Y \to X$ is the blow-up of $Y$ along $Z$. In particular, we have a cartesian square, for every $i\in \{1,\dots,n\}$,
\begin{equation}
\label{squ}
\begin{gathered}
\xymatrix{
\widetilde{Y_i}\ar[r]^{g_i} \ar[d]_{u_i} & Y_i \ar[d]^{l_i} \\ 
\widetilde{X} \ar[r]^f & X
}
\end{gathered}
\end{equation}

The set theoretic image of $Z\to X$ is $\Sigma$ (this is where we use that each $Z_i$ is zero-dimensional), hence $W$ is supported on the closure of $\Sigma$, which is contained in $\Sing X$. Therefore $f$ is an isomorphism outside of $\Sing X$, and in particular is birational. It is projective, as is any blow-up. Moreover $\widetilde{X}$ is reduced, being the blow-up of a reduced scheme. 

Thus, to conclude the proof in the normal case, it will suffice to verify the condition \ref{cond:reg}. So let $x$ be a point of $\widetilde{X}$ such that $f(x)$ has codimension $\leq 2$ in $X$. If $f(x) \not \in \Sigma$, then $f(x) \not \in \Sing X$ by construction of $\Sigma$. Therefore $f$ induces an isomorphism $\Oc_{X,f(x)} \to \Oc_{\widetilde{X},x}$, and it follows that $\widetilde{X}$ is regular at $x$. Now assume that $f(x) =s_i$ for some $i \in \{1,\dots,n\}$. Let $F$ be the residue field of $\widetilde{X}$ at $x$. Composing the $F$-point given by $x$ with the morphism $f$, we obtain an $F$-point of $X$. Since its image is the point $s_i$, this $F$-point factors through the localisation $l_i\colon Y_i \to X$. In view of the cartesian square \eqref{squ}, we obtain an $F$-point $y$ of $\widetilde{Y_i}$ such that $x=u_i(y)$. Since $\widetilde{Y_i}$ is regular at $y$, and $u_i$ is flat, it follows from \cite[(6.5.2, (i))]{ega-4} that $\widetilde{X}$ is regular at $x$.\\

Now let $X$ be an arbitrary reduced quasi-excellent scheme. Its normalisation $\pi \colon X' \to X$ is finite and birational. Applying the above construction to the quasi-excellent normal scheme $X'$, we obtain a projective birational morphism $f' \colon \widetilde{X} \to X'$ satisfying the conditions \ref{cond:reg} and \ref{cond:isom}. The morphism $f=\pi \circ f'$ is projective and birational. Since $\pi$ induces an isomorphism outside of $\Sing X$, and $f'$ induces an isomorphism outside of $\Sing X' \subset \pi^{-1} \Sing X$, we see that $f$ satisfies \ref{cond:isom}. To see that $f$ satisfies \ref{cond:reg}, let $x$ be a point of $\widetilde{X}$ such that $f(x)$ has codimension $\leq 2$ in $X$. By \autoref{lemm:codim} below applied to the morphism $\pi$, the point $f'(x)$ has codimension $\leq 2$ in $X'$. Since $f'$ satisfies \ref{cond:reg}, it follows that $\widetilde{X}$ is regular at $x$.
\end{proof}

\begin{remark}
From the proof, we see that we may take for $f$ the normalisation of $X$, followed by the blow-up along a closed subscheme of codimension two.
\end{remark}

We will be using the following variant of \autoref{th:res}.
\begin{corollary}
\label{cor:resqproj}
Let $X$ be a reduced quasi-excellent scheme, of finite type over a scheme $S$. Then there is a projective birational morphism $f \colon X' \to X$, with $X'$ reduced and quasi-projective over $S$, and regular at any of his points $x$ such that $f(x)$ has  codimension $\leq 2$ in $X$.
\end{corollary}
\begin{proof}
Applying Chow's lemma \cite[(5.6.1)]{ega-2}, we obtain a projective birational morphism $Y \to X$ with $Y$ reduced and quasi-projective over $S$. By \autoref{lemm:codim}, a point has codimension $\leq 2$ in $Y$ as soon as its image has codimension $\leq 2$ in $X$. The statement follows by applying \autoref{th:res} to the scheme $Y$.
\end{proof}

\begin{lemma}
\label{lemm:codim}
Let $f\colon Y \to X$ be a birational morphism of finite type, and $y$ a point of $Y$ with image $x=f(y)$ in $X$. Then we have $\codim_Y(y) \leq \codim_X(x)$.
\end{lemma}
\begin{proof}
Let $Y_0$ be an irreducible component of $Y$ containing $y$ and such that $\codim_Y(y) = \codim_{Y_0}(y)$. There is a unique irreducible component $X_0$ of $X$ such that $f$ induces a birational morphism $f_0\colon Y_0 \to X_0$ of finite type. Applying the dimension formula \cite[(5.6.5.1)]{ega-4} to the morphism $f_0$, we obtain $\codim_{Y_0}(y) \leq \codim_{X_0}(x)$, and we conclude since $\codim_{X_0}(x) \leq \codim_{X}(x)$.
\end{proof}

\section{Detection by regular schemes}
\label{sect:procedure}
Let $T\to S$ be a morphism of finite type, and $F$ a functor $\Scht \to \Ab$. When $f$ is a morphism in $\Scht$, we will write $f_*$ for $F(f)$. We consider the following conditions on $F$, where all schemes are of finite type over $T$ and all morphisms are over $T$.
\begin{enumerate}[label=(F\arabic{*}), ref=(F\arabic{*})]
\item \label{F:open} For every open immersion $u\colon U \to X$, we have a morphism of abelian groups $u^*\colon F(X) \to F(U)$.

\item \label{F:loc} If $i$ is a closed embedding, and $u$ the open immersion of its complement, we have $\ker u^* \subset \im i_*$.

\item \label{F:cart} For any cartesian square, with $u$ an open immersion and $f$ a proper morphism
\[ \xymatrix{
V\ar[r]^v \ar[d]_g & Y \ar[d]^f \\ 
U \ar[r]^u & X
}\]
we have $u^* \circ f_* = g_*\circ v^*$.

\item \label{G:fund} If $X$ is integral or empty, we have a specified element $1_X \in F(X)$ such that:
\begin{enumerate}[label=(F4\alph{*}), ref=(F4\alph{*})]
\item \label{F:fundopen} If $u\colon U \to X$ is an open immersion, then $u^*(1_X)=1_U$.
\item \label{F:isom} If $f\colon Y \to X$ is an isomorphism, then $f_*(1_Y)=1_X$.

\end{enumerate}
\item \label{F:dim} The group $F(X)$ is generated by the elements $f_*(1_Y)$, where $Y$ is integral, the morphism $f\colon Y \to X$ is proper, and $\dim_S Y \leq \dim_S X$ (see \ref{conv:reldim}).
\end{enumerate}

\begin{lemma}
\label{lemm:test1}
Let $T$ be a scheme of finite type over a quasi-excellent scheme $S$. Let $G,H$ be two functors $\Scht \to \Ab$ satisfying \ref{F:open} and \ref{F:loc}. Assume that $G$ additionally satisfies \ref{F:cart}, \ref{G:fund} and \ref{F:dim}. Let $d$ be an integer such that, for every scheme $X$ of finite type over $T$, the following conditions are satisfied.
\begin{enumerate}[label=(\roman{*}), ref=(\roman{*})]
\item \label{G:gen} The group $G(X)$ is generated by the elements $f_*(1_Y)$, where $Y$ is integral, the morphism $f\colon Y \to X$ is proper, and $\dim_S Y \leq d$.

\item \label{H:dim} When $\dim_S X < d-2$, we have $H(X)=0$.
\newcounter{saveenum}\setcounter{saveenum}{\value{enumi}}
\end{enumerate}

Assume that, for every scheme $X$ of finite type over $T$, we have a morphism $\rho_X\colon G(X) \to H(X)$ satisfying the following conditions.
\begin{enumerate}[label=(\roman{*}), ref=(\roman{*})]
\setcounter{enumi}{\value{saveenum}}
\item \label{rho:proper} If $f\colon Y \to X$ is a proper morphism, then $f_* \circ \rho_Y=\rho_X \circ f_*$.

\item \label{rho:open} If $u\colon U \to X$ is an open immersion, then $u^* \circ \rho_X=\rho_U \circ u^*$.

\item \label{rho:regular} For any regular $T$-scheme $R$, quasi-projective over $S$, we have $\rho_R(1_R)=0$.
\end{enumerate}
Then $\rho_X=0$ for every scheme $X$ of finite type over $T$.
\end{lemma}
\begin{proof}
Assuming the contrary, there is by \ref{G:gen} and \ref{rho:proper} an integral scheme $X$ of finite type over $T$ with $\dim_S X \leq d$ and $\rho_X(1_X)\neq0$. Using \autoref{cor:resqproj}, we find a proper birational morphism $f\colon X' \to X$, with $X'$ integral and quasi-projective over $S$, and such that $f$ maps singular points of $X'$ to codimension $\geq 3$ in $X$. By quasi-excellence of $X'$, its singular locus is closed; we let $i$ be the embedding of the associated reduced closed subscheme of $X'$, and $r \colon R \to X'$ be its open complement. Using successively \ref{rho:open}, \ref{F:fundopen} and \ref{rho:regular} we have
\[
r^*\circ \rho_{X'}(1_{X'}) = \rho_R\circ r^*(1_{X'}) = \rho_R(1_R)=0.
\]
By \ref{F:loc} with $F=H$, this implies that $\rho_{X'}(1_{X'}) \in \im i_*$. Let $z \colon Z \hookrightarrow X$ be the closed embedding of the scheme theoretic image of $f\circ i$. Since $Z$ has codimension $\geq 3$ in $X$, we have $\dim_S Z \leq \dim_S X - 3 < d-2$ by \eqref{eq:dimS}. Therefore $H(Z)=0$ by \ref{H:dim}, hence using \ref{rho:proper}
\[
\rho_X \circ f_*(1_{X'}) = f_* \circ \rho_{X'}(1_{X'}) \in \im (f\circ i)_* \subset \im z_*=0.
\]

Now let $U$ be a non-empty open subscheme of $X$ over which $f$ is an isomorphism, and form the square of \ref{F:cart}. By \ref{F:cart}, \ref{F:fundopen} and \ref{F:isom} we have
\[
u^*(f_*(1_{X'}) - 1_X) = u^* \circ f_*(1_{X'}) - u^*(1_X)=g_* \circ v^*(1_{X'}) - 1_U = g_*(1_V) -1_U =0.
\]
Hence by \ref{F:loc} with $F=G$ we have $f_*(1_{X'}) - 1_X \in \im j_*$, where $j \colon Q \hookrightarrow X$ is the closed embedding of a complement of $u$. Since $\rho_X(1_X) \neq 0$ and $\rho_X \circ  f_*(1_{X'}) = 0$, we have $\rho_X \circ j_* \neq 0$. It follows from \ref{F:dim} and \ref{rho:proper} that we may find a proper morphism $Y \to Q$ with $Y$ integral and $\dim_S Y \leq \dim_S Q < \dim_S X$ (the last inequality follows from \eqref{eq:dimS}), such that $\rho_Y(1_Y) \neq 0$.

Thus we construct by induction (letting $X_{-1}=X$) for every $n \in \mathbb{N}$ a proper morphism $X_n \to X_{n-1}$, with $X_n$ integral, $\dim_S X_n < \dim_S X_{n-1}$, and $\rho_{X_n}(1_{X_n}) \neq 0$. The images of $X_n$ in $S$ form a decreasing sequence of closed subsets whose codimensions tend to infinity. This is impossible since $S$ is noetherian.
\end{proof}
\begin{remark}
\label{rem:Q}
Assume that $S$ is a $\Qq$-scheme. Using resolution of singularities for $\Qq$-schemes \cite{Temkin-desing} instead of \autoref{th:res}, we can remove the conditions \ref{H:dim} and \ref{rho:open} in the statement of \autoref{lemm:test1}.
\end{remark}

We now describe how the Chow group and the Grothendieck group (see \ref{conv:chow}) give rise to the functors $G$ and $H$ of \autoref{lemm:test1}, when $S$ is regular.
 
\begin{example}
\label{ex:GCH}
Let $C$ be an abelian group and $j$ an integer. The functor $G=H=\CH_j(-)\otimes_{\Zz} C$ satisfies the conditions \ref{F:open} to \ref{F:dim} (with $1_X=[X]$ when $\dim_S X=j$ and $1_X=0$ otherwise). Moreover the condition \ref{G:gen}, resp.\ \ref{H:dim}, of \autoref{lemm:test1} is satisfied when $j \leq d$, resp.\ $j \geq d-2$.
\end{example}

\begin{example}
\label{ex:Hreldim}
We say that a morphism $P \to T$ of finite type \emph{has relative dimension $\leq r$} if for every scheme $Y$ of finite type over $T$, we have $\dim_S (P \times_T Y) - \dim_S Y \leq r$. Examples of such morphisms are given by closed embeddings ($r=0$), and flat morphisms of constant relative dimension $r$.

Let $j$ be an integer, and $P \to T$ a morphism of relative dimension $\leq r$. For a scheme $X$ of finite type over $T$, we let $H(X)=\CH_{j+r}(X \times_T P)$. This gives a functor $H \colon \Scht \to \Ab$ satisfying \ref{F:open}, \ref{F:loc}, and the condition \ref{H:dim} of \autoref{lemm:test1} when $j\geq d-2$.
\end{example}

\begin{example}
\label{ex:GK}
The functor $G=\K_0'(-)_{(d)}$ (see \ref{conv:chow}) satisfies the conditions \ref{F:open} to \ref{F:dim} if we let $1_X=[\Oc_X]$ when $\dim_S X \leq d$, and $1_X=0$ otherwise. In addition, the condition \ref{G:gen} of \autoref{lemm:test1} is satisfied.
\end{example}

\section{Integrality of the Chern character}
\label{sect:integrality}

We formulated in \cite{firstst} the following conjecture, which depends on a prime number $p$ and an integer $n$ (see \ref{conv:abelian} and \ref{conv:chow} for the relevant definitions).
\begin{conjecture}
\label{conj}
Let $S$ be a regular scheme, $X$ a scheme of finite type over $S$, $d$ an integer, and $x \in \K_0'(X)_{(d)}$. We have, for all $i\leq n$,
\[
v_p(\ch_{d-i}x) \geq -\Big[ \frac{i}{p-1}\Big].
\]

\end{conjecture}
We proved \autoref{conj} when $n\leq p(p-1)-1$ and $S$ is of finite dimension in \cite[Proposition~3.1 and Appendix]{firstst}.

\begin{theorem}
\label{th:integrality}
\autoref{conj} is true for $p=2$ and $n=2$, when $S$ is excellent.
\end{theorem}
\begin{proof}
We take $G$ as in \autoref{ex:GK}, and $H$ as in \autoref{ex:GCH} (with $j=d-i$ and $C=\Qq/(2^{-2}\cdot \Zz_{(2)})$, see \ref{conv:abelian}), and check the conditions of \autoref{lemm:test1} for $T=S$ and the operation
\[
\rho_X=\ch_{d-i} \colon G(X)=\K'_0(X)_{(d)} \to H(X) = \CH_{d-i}(X) \otimes_{\Zz} \Qq/\{y | v_2(y) \geq -2\}.
\]
The conditions \ref{rho:proper} and \ref{rho:open} follow from \cite[Theorem~18.3 (1), (4) and \S20]{Ful-In-98}, while \ref{rho:regular} is \autoref{lemm:reg} below.
\end{proof}

\begin{lemma}
\label{lemm:reg}
Let $X$ and $S$ be regular schemes, with $X$ quasi-projective over $S$. Then for any integer $i$ and any prime number $p$, we have
\[
v_p(\ch_{\dim_S X-i}[\Oc_X]) \geq -\Big[ \frac{i}{p-1}\Big].
\]
\end{lemma}
\begin{proof}
We apply \autoref{lemm:decomp_qproj}. The lemma can then be proved using the Riemann-Roch theorem of \cite[\S20]{Ful-In-98}; see the proof of \cite[Proposition~4.1]{firstst}.
\end{proof}

\begin{lemma}
\label{lemm:decomp_qproj}
Let $S$ be a regular scheme. Any quasi-projective morphism $X \to S$ decomposes as $p \circ i$, with $i$ a closed embedding, and $p$ a smooth quasi-projective morphism. If $X$ is regular, then $i$ is a regular closed embedding.
\end{lemma}
\begin{proof}
By \cite[II, Corollaire~2.2.7.1]{sga6}, the scheme $S$, being separated and regular, is divisorial (i.e.\ possesses an ample family). It follows from \cite[Proposition~2.2.3, Lemme~2.1.1~a)]{sga6} that any coherent sheaf of $\Oc_S$-modules is a quotient of a locally free coherent sheaf. Thus any projective morphism with target $S$ factors as a closed immersion into a projective bundle over $S$. In particular $X \to S$ decomposes as $X \xrightarrow{u} Z \xrightarrow{j} P \xrightarrow{q} S$, with $u$ an open immersion, $j$ a closed embedding, and $q$ a smooth projective morphism. Let $T$ be a closed complement of $X$ in $Z$, and $W$ the open complement of $j(T)$ in $P$. Then $j$ restricts to a closed embedding $i \colon X \hookrightarrow W$, and the composite $p\colon W \to P \to S$ is smooth and quasi-projective.

The last statement follows from \cite[(19.1.1)]{ega-4}, since $W$ is regular, being smooth over a regular scheme \cite[(17.5.8, (iii))]{ega-4}.
\end{proof}

\begin{theorem}
\label{cor:three}
\autoref{conj} is true for $p=2$ and $n=3$, when $S$ is excellent and of finite dimension.
\end{theorem}
\begin{proof}
Let $X' \to X$ be a Chow envelope over $S$ (see \ref{def:envelope}). Since the morphism $\K_0'(X')_{(d)} \to \K_0'(X)_{(d)}$ is surjective, we may replace $X$ with $X'$, and assume that $X$ is quasi-projective over $S$. We may also easily reduce to the case when $S$ is connected. Let $x \in \K_0'(X)_{(d)}$, and write $\alpha = \psi_{-1}(x) - (-1)^{-d}\cdot x$. Here $\psi_{-1}$ is the $(-1)$-st homological Adams operation (corresponding to duality theory), see \cite[Th\'eor\`eme~7]{Sou-Op-85} where $S$ is assumed to be of finite dimension. We have in $\CH_{d-3}(X) \otimes_\Zz \Qq$, by \cite[Proposition~2.4 and Appendix]{firstst},
\begin{align*} 
 \ch_{d-3}\alpha &= \ch_{d-3}( \psi_{-1}(x) - (-1)^d\cdot x ) \\ 
 &= (-1)^{3-d} \cdot \ch_{d-3}x -  (-1)^{-d}\cdot \ch_{d-3}x\\
 &= 2 \cdot (-1)^{1-d} \cdot \ch_{d-3}x.
 \end{align*}
Since $\alpha \in \K_0'(X)_{(d-1)}$ by \cite[Proposition~A.1]{firstst}, we have $v_2(\ch_{d-3}\alpha) \geq -2$ by \autoref{th:integrality}. It follows that $v_2(\ch_{d-3}x) \geq -3$, as required.
\end{proof}

\begin{remark}
The arguments used in the proof of \autoref{cor:three} are those of the proof of \cite[Theorem~3.2]{firstst}. More generally (starting the induction $d=n$ instead of $d=0$ in the proof of \cite[Theorem~3.2]{firstst}):

\emph{Fix a prime number $p$. Assume that \autoref{conj} is true for $n=m$, with $S$ of finite dimension. Let $m'$ be the smallest multiple of $p(p-1)$ such that $m' >m$. Then \autoref{conj} is true for $n=m'-1$.}
\end{remark}

\begin{remark}
Using \autoref{rem:Q} and proceeding as in the proof of \autoref{th:integrality}, we see that \autoref{conj} is true for any $p$ and $n$ when $S$ is an excellent $\Qq$-scheme.
\end{remark}

Let $S$ be a regular excellent scheme. By \cite[Section~5]{firstst}, \autoref{th:integrality} and \autoref{cor:three}, we obtain morphisms of functors $\Sch \to \Ab$ 
\[
T_2 \colon \CH_\bullet(-) \otimes_{\Zz} \Zz/2 \to \big(\im(\CH_{\bullet-2}(-) \otimes_{\Zz} \Zz_{(2)} \to \CH_{\bullet-2}(-)\otimes_{\Zz} \Qq)\big) \otimes_{\Zz} \Zz/2,
\]
and when $S$ is of finite dimension,
\[
T_3 \colon \CH_\bullet(-) \otimes_{\Zz} \Zz/2 \to \big(\im(\CH_{\bullet-3}(-) \otimes_{\Zz} \Zz_{(2)} \to \CH_{\bullet-3}(-)\otimes_{\Zz} \Qq)\big) \otimes_{\Zz} \Zz/2.
\]

The operation $T_2$ is the second Steenrod square (modulo two-torsion), see \cite[Proposition~8.1]{firstst}. By contrast, the operation $T_3$ is not the usual third Steenrod square; nonetheless it can be used (together with $T_2$ and the first Steenrod square $T_1$) to define the third reduced Steenrod square, see \cite[Remark~8.6]{firstst}.

Proceeding as in \cite[Section~6]{firstsq}, but using operation $T_2$ instead of $T_1$, on can prove the following statement, which was already known in characteristic not two (see \cite{Kar-On-2003}).
\begin{theorem}
Let $\varphi$ be an anisotropic non-degenerate quadratic form over an arbitrary field. Let $\iu$ be the first Witt index of $\varphi$. If $\dim \varphi - \iu $ is not divisible by $4$, then $\iu$ is equal to $1$ or $2$.
\end{theorem}

\section{Adem relation for the first Steenrod square}
\label{sect:adem}
In \cite[\S7]{duality} we gave a construction of the first Steenrod square, an operation $\Sq\colon\CH_\bullet(X) \otimes_{\Zz} \Zz/2 \to\CH_{\bullet-1}(X) \otimes_{\Zz} \Zz/2$, for $X$ quasi-projective over a field.\footnote{This operation is a homological analog of $\Squ$ in topology; the Bockstein operation in motivic cohomology vanishes on Chow groups.} We were not able to prove the relation $\Sq \circ \Sq = 0$, and instead proved weaker relations. In this section, we use \autoref{th:res} to prove the relation $\Sq \circ \Sq = 0$.\\

We first extend the definition of $\Sq$ to arbitrary schemes $X$ of finite type over a field. The next definition is close to that of a site, but will be more convenient for us, since our covers will always be singletons.
\begin{definition}
Let $\Cc$ be a category with fibre products. We say that $\Cc$ is a \emph{category with covers} if are given a class of morphisms in $\Cc$, called \emph{covers}, satisfying the following conditions.
\begin{itemize}[label=---]
\item If $Y \to X$ is a cover and $X' \to X$ is in $\Cc$, then $X'\times_X Y \to X'$ is a cover.

\item If $Z\to Y$ and $Y \to X$ are covers, so is the composite $Z \to X$.
\end{itemize}

For a category $\Dc$, we denote by $\Hom(\Dc,\Ab)$ the category of functors $\Dc \to \Ab$. Let $\Cc$ be a category with covers. The category $\coshv(\Cc)$ is the full subcategory of $\Hom(\Cc,\Ab)$, consisting of those $F \colon \Cc \to \Ab$, called \emph{abelian cosheaves on $\Cc$}, such that for any cover $\pi\colon Y \to X$ the sequence of abelian groups
\[
F(Y \times_X Y) \xrightarrow{F(\pi_1) - F(\pi_2)} F(Y) \xrightarrow{F(\pi)} F(X) \to 0,
\]
where $\pi_1,\pi_2 \colon Y \times_X Y \to Y$ are the two projections, is exact.
\end{definition}

\begin{proposition}
\label{prop:cosheaves}
Let $\Cc$ be a category with covers, and $\Dc$ a full subcategory. Assume that for every object $X \in \Cc$, there is a cover $Y\to X$ with $Y \in \Dc$. Then the natural functor $r \colon \coshv(\Cc) \to \Hom(\Dc,\Ab)$ is fully faithful.
\end{proposition}
\begin{proof}
Let $\varphi\colon F\to G$ be a morphism in $\coshv(\Cc)$, such that $r(\varphi)=0$. Let $X \in \Cc$. Choose a cover $\pi \colon Y \to X$ with $Y \in \Dc$. We have a commutative diagram
\[ \xymatrix{
F(Y)\ar[rr]^{F(\pi)} \ar[d]_{\varphi_Y} && F(X) \ar[d]^{\varphi_X} \\ 
G(Y) \ar[rr]^{G(\pi)} && G(X)
}\]
and $\varphi_Y=r(\varphi)_Y=0$, so that $\varphi_X \circ F(\pi)=0$. Since $F(\pi)$ is an epimorphism, it follows that $\varphi_X=0$. This proves that $\varphi=0$, hence $r$ is faithful.

Now let $F,G \in \coshv(\Cc)$ and $\varphi\colon r(F)\to r(G)$ be a morphism of functors $\Dc \to \Ab$. Let $X \in \Cc$. Choose a cover $\pi\colon Y \to X$ with $Y \in \Dc$, and  a cover $Z \to Y \times_X Y$ with $Z \in \Dc$. The two morphisms $p_1,p_2 \colon Z \to Y \times_X Y \to Y$ in $\Cc$ belong to $\Dc$, since $\Dc$ is a full subcategory of $\Cc$. Therefore we have the commutative diagram with solid arrows
\[ \xymatrix{
F(Z)\ar[rrr]^{F(p_1)-F(p_2)} \ar[d]^{\varphi_Z} &&& F(Y) \ar[d]^{\varphi_Y} \ar[rr]^{F(\pi)}&& F(X) \ar[r] \ar@{-->}[d]^{\phi_X^\pi}&0\\ 
G(Z) \ar[rrr]^{G(p_1)-G(p_2)} &&& G(Y)\ar[rr]^{G(\pi)}&& G(X) \ar[r] &0
}\]
Since $F$ is an abelian cosheaf, the morphism $F(Z) \to F(Y \times_X Y)$ is surjective, and moreover the upper row is exact. So is the lower row, by the same argument. Thus there is a unique morphism $\phi_X^\pi$ fitting in the diagram.

For $i \in \{1,2\}$, let $\pi_i \colon Y_i \to X$ be a cover with $Y_i \in \Dc$. Choose a cover $Y' \to Y_1 \times_X Y_2$, with $Y' \in \Dc$. For $i \in \{1,2\}$, the induced morphism $q_i\colon Y' \to Y_i$ is in $\Dc$, and $\pi_i \circ q_i$ is a cover, and moreover we have two commutative diagrams
\[ \xymatrix{
F(Y')\ar[rr]^{F(q_i)} \ar[d]_{\varphi_{Y'}}&& F(Y_i) \ar[d]_{\varphi_{Y_i}} \ar[rr]^{F(\pi_i)} && F(X)\ar[d]^{\phi_X^{\pi_i}}&&F(Y') \ar[rr]^{F(\pi_i \circ q_i)} \ar[d]_{\varphi_{Y'}} && F(X)\ar[d]^{\phi_X^{\pi_i\circ q_i}}\\ 
G(Y') \ar[rr]^{G(q_i)}&& G(Y_i) \ar[rr]^{G(\pi_i)} && G(X)&&G(Y') \ar[rr]^{G(\pi_i \circ q_i)} && G(X)
}\]
By unicity of $\phi^\pi_X$, we see that $\phi_X^{\pi_i}=\phi_X^{\pi_i\circ q_i}$ for $i\in \{1,2\}$. Since the morphism $\pi_i\circ q_i$ is independent of $i$, so is the morphism $\phi_X^{\pi_i}$. This gives a well-defined morphism $\phi_X=\phi_X^\pi\colon F(X) \to G(X)$ for every $X \in \Cc$. When $X\in \Dc$, we can take $\pi=\id_X$, so that $\phi_X=\varphi_X$.

Now let $f \colon X' \to X$ be a morphism in $\Cc$. Let $\pi\colon Y \to X$ be a cover with $Y\in \Dc$, and $Y' \to Y\times_X X'$ a cover with $Y' \in \Dc$. The morphism $\pi'\colon Y' \to X'$ is a cover, and $g \colon Y' \to Y$ is in $\Dc$. Then 
\begin{align*} 
G(f)\circ \phi_{X'} \circ F(\pi')  &= G(f)\circ G(\pi') \circ \varphi_{Y'}\\ 
 &= G(\pi) \circ G(g) \circ \varphi_{Y'}\\
 &= G(\pi) \circ \varphi_Y \circ F(g)\\
 &= \phi_X \circ F(\pi) \circ F(g)\\
 &= \phi_X \circ F(f) \circ F(\pi').
 \end{align*}
Since $F(\pi')$ is an epimorphism, we get that $G(f)\circ \phi_{X'} = \phi_X \circ F(f)$. Thus $\phi$ defines a morphism $F \to G$ in $\coshv(\Cc)$, such that $r(\phi)=\varphi$. Therefore $r$ is full, concluding the proof.
\end{proof}

\begin{proposition}
\label{prop:kimura}
Let $C$ be an abelian group, $i$ an integer, and $S$ the spectrum of a field. The envelopes (see \ref{def:envelope}) define a class of covers on the category $\Sch$, for which the functor $\CH_i(-)\otimes_{\Zz} C$ an abelian cosheaf.
\end{proposition}
\begin{proof}
The first statement follows from \cite[Lemma~18.3 (1), (2)]{Ful-In-98}. When $C=\Zz$ the second statement is \cite[Theorem~1.8]{Kimura-Fractional} (this does not use resolution of singularities). In general, it follows from the right-exactness of $-\otimes_\Zz C$.
\end{proof}

\begin{proposition}
\label{cor:extsq}
Let $S$ be the spectrum of a field. The operation $\Sq\colon\CH_\bullet(X) \otimes_{\Zz} \Zz/2 \to\CH_{\bullet-1}(X) \otimes_{\Zz} \Zz/2$ defined in \cite[\S7]{duality} for quasi-projective $S$-schemes $X$ can be extended in a unique way compatible with proper push-forwards to all schemes $X$ of finite type over $S$. Moreover the resulting operation, which we denote again $\Sq$, is compatible with restrictions to open subschemes.
\end{proposition}
\begin{proof}
By \autoref{prop:kimura} with $C=\Zz/2$ and existence of Chow's envelopes (see \ref{def:envelope}), the conditions of \autoref{prop:cosheaves} are satisfied for $\Cc=\Sch$ with covers given by envelopes and $\Dc$ the full subcategory of quasi-projective $S$-schemes. Therefore the operation $\Sq$ of \cite[\S7]{duality} extends uniquely to a morphism $\Sq$ of functors $\Sch \to \Ab$.

Let $u \colon U \to X$ be an open immersion. Let $\pi \colon Y \to X$ be a Chow envelope over $S$, and form the cartesian square
\[ \xymatrix{
V\ar[r]^v \ar[d]_\tau & Y \ar[d]^\pi \\ 
U \ar[r]^u & X
}\]
Then $v$ is an open immersion between quasi-projective $S$-schemes, hence $\Sq \circ v^*=v^* \circ \Sq$ by \cite[Proposition~7.6]{duality}. We have
\[
\Sq \circ u^* \circ \pi_* = \Sq \circ \tau_* \circ v^*=\tau_* \circ \Sq \circ v^*=\tau_* \circ v^* \circ \Sq=u^* \circ \pi_* \circ \Sq=u^* \circ \Sq \circ \pi_*.
\]
Since $\pi_*$ is an epimorphism, it follows that $\Sq \circ u^*=u^* \circ \Sq$, hence $\Sq$ is compatible with restrictions to open subschemes.
\end{proof}

\begin{lemma}
\label{lemm:regsq}
Let $X$ be a regular scheme, quasi-projective over a field. Then 
\[
\Sq \circ \Sq[X]=0 \in \CH(X) \otimes_{\Zz} \Zz/2.
\]
\end{lemma}
\begin{proof}
Let $\Tan_X \in \K_0(X)$ be the virtual tangent bundle of $X$, and $L_X$ a line bundle over $X$ representing the image of $\Tan_X$ under the group homomorphism $\det \colon \K_0(X) \to \Pic(X)$ (the sheaf of sections of $L_X$ is the canonical sheaf of $X$). Let $s \colon X \hookrightarrow L_X$ be its zero section. Then the first Chern class operators satisfy:
\[
c_1(\Tan_X) = c_1(L_X) = s^* \circ s_* \colon \CH(X) \to \CH(X).
\]
By \cite[(11)]{duality}, we have $\Sq[X] = c_1(\Tan_X)[X]$ in $\CH(X) \otimes_{\Zz} \Zz/2$. Using \cite[Proposition~7.6]{duality}, we have in $\CH(X) \otimes_{\Zz} \Zz/2$
\begin{align*} 
\Sq \circ \Sq[X]   &= \Sq \circ s^* \circ s_*[X]  \\ 
 &= s^* \circ \Sq \circ s_*[X] + c_1(L_X) \circ s^* \circ s_*[X]\\
 &= s^* \circ s_* \circ \Sq[X] + c_1(L_X) \circ s^* \circ s_*[X]\\
 &= s^* \circ s_* \circ c_1(L_X)[X] + c_1(L_X) \circ s^* \circ s_*[X]\\
 &= c_1(L_X)\circ c_1(L_X)[X] + c_1(L_X)\circ c_1(L_X)[X]\\
 &=0 \mod 2.\qedhere
 \end{align*}
\end{proof}

\begin{proposition}
We have $\Sq \circ \Sq=0$.
\end{proposition}
\begin{proof}
We apply \autoref{lemm:test1} with $S=T$ the spectrum of a field, and with $G$, resp.\ $H$, given by \autoref{ex:GCH}, with $C=\Zz/2$ and $j=d$, resp.\ $j=d-2$, for the operation $\rho_X = \Sq \circ \Sq\colon \CH_d(X) \otimes_{\Zz} \Zz/2\to\CH_{d-2}(X) \otimes_{\Zz} \Zz/2$. Then the conditions \ref{rho:proper} and \ref{rho:open} follow from \autoref{cor:extsq}, and \ref{rho:regular} from \autoref{lemm:regsq}.
\end{proof}

\section{Bivariant Chow group}
\label{sect:bivariant}
In this section we consider the bivariant Chow group $\A(X \to T)$. Of special interest is the group $\A(\id_X)$, which can be considered as the universal cohomology theory acting on Chow groups (it has a ring structure and is contravariant for arbitrary morphisms of finite type). More generally when $X$ is a closed subscheme of $T$, the non-unital ring $\A(X \to T)$ corresponds to cohomology of $T$ with supports in $X$. We are interested in the commutativity of this ring, and in Poincar\'e duality-type statements.\\

We first define bivariant Chow groups following \cite[\S17, \S20]{Ful-In-98}, with the difference that we require bivariant classes to be compatible with smooth pull-backs, as opposed to arbitrary flat pull-backs (see however \autoref{rem:commt} \ref{rem:flat} and \autoref{th:Q} \ref{Q:dim}).
 
Let $S$ be a regular scheme, and $X \to T \to S$ morphisms of finite type. For $m\in \Zz$, we consider the set $\B^m(X \to T)$ of morphisms $\alpha$ of functors $\Scht \to \Ab$
\[
\CH_\bullet(-) \to \CH_{\bullet-m}(- \times_T X), \quad x \mapsto \alpha \cap x.
\]
Let $Y$ be a scheme of finite type over $T$. Any element $\alpha \in \B^m(X \to T)$ induces an element of $\B^m(X\times_T Y \to Y)$ which we still denote by $\alpha$. When $g \colon X \to T$ is a flat morphism of constant relative dimension $-r$, or a regular closed embedding of codimension $r$, there is an orientation class $[g] \in \B^r(X \to T)$, corresponding to the fact that base changes of such morphisms admit functorial (refined) pull-backs (see \cite[\S17.4]{Ful-In-98}). There is a product $\B^m(X \to T) \otimes_{\Zz} \B^n(Y \to T) \to \B^{m+n}(X \times_T Y \to T)$, written $a\otimes b \mapsto a\cdot  b$. 

The group $\A^m(X \to T)$ consists of the elements $\alpha \in \B^m(X \to T)$, called \emph{bivariant classes}, such that $[f]\cdot \alpha=\alpha \cdot [f]$ for $f$ any smooth morphism or regular closed embedding. We denote by $\A(X \to T)$ the direct sum of the groups $\A^m(X \to T)$ over the integers $m \in \Zz$. If $X \to T$ has relative dimension $\leq r$ (\autoref{ex:Hreldim}), then $\B^m(X \to T)=\A^m(X \to T)=0$ for $m <- r$.

The product of bivariant classes is a bivariant class, and so is $[g]$ for $g$ as above (flat or regular closed embedding). When $h \colon Y \to X$ is a morphism in $\Scht$, there is a morphism $h_* \colon \A^m(Y \to T) \to \A^m(X \to T)$.

\begin{remark}
\label{rem:divisor}
An element $\alpha \in \B^m(X \to T)$ such that $[f] \cdot \alpha = \alpha \cdot [f]$, for $f$ any smooth morphism or effective Cartier divisor with trivial normal bundle, is a bivariant class. This can be seen using deformation to the normal cone (see the proof of \cite[Theorem~17.1]{Ful-In-98}).
\end{remark}

\begin{example}
\label{ex:localised}
Let $X \to T$ be a closed embedding, and $\Ec_{\bullet}$ a bounded complex of locally free coherent $\Oc_T$-modules, which is exact off $X$. The $n$-th localised Chern class $c_n(\Ec_{\bullet})$ is an element of $\A^n(X \to T)$, see \cite[Example~18.1.3, \S20]{Ful-In-98}.
\end{example}

\begin{example}
\label{ex:chowcoh}
Let $X \to T$ be a closed embedding. Let $\Ko{n}{T}$ be the Zariski sheaf on $T$ associated with the $n$-th Quillen $K$-group of vector bundles. We describe a morphism
\[
\HK{n}{X}{T} \to \A^n(X \to T).
\]
For any scheme $Y$ of finite type over $T$, there is a morphism (see \cite[\S7, \S8]{Gil-Ri-81})
\[
\HK{n}{X}{T} \otimes \CH_d(Y) \to \CH_{d-n}(X \times_T Y).
\]
The projection formula \cite[Theorem~8.8]{Gil-Ri-81} says that it is compatible with proper push-forwards; therefore any element $\alpha \in \HK{n}{X}{T}$ defines an element of $\B^n(X \to T)$, which we also denote by $\alpha$. When $f$ is smooth (or more generally flat of constant relative dimension), the formula $[f] \cdot \alpha=\alpha \cdot [f]$ is easily verified. 

To show that $\alpha$ defines an element of $\A^n(X \to T)$, it will suffice by \autoref{rem:divisor} to consider an effective Cartier divisor $f \colon D \to T$ with trivial normal bundle, and prove that $\alpha \cdot [f]=[f] \cdot \alpha$. The invertible sheaf $\Oc_T(D)$, being trivial off $D$, gives rise to an element $\delta \in H_D^1(T,\Oc_T^\times)=\HK{1}{D}{T}$.

We claim that $[f] = \delta$. Indeed it is enough to check that $f^![Z] = \delta \cap [Z] \in \CH(Z)$ when $Z$ is an integral scheme of finite type over $T$. Assume first that $Z \to T$ factors through $f$. Since $f$ has trivial normal bundle, $\delta$ restricts to zero in $\HK{1}{D}{D}$, hence we have $\delta \cap [Z]=0$. On the other hand, $f^![Z]=0$ by e.g.\ \cite[Proposition~2.3 (e)]{Ful-In-98}. Otherwise $f$ induces an effective Cartier divisor with trivial normal bundle on $Z$, and we are reduced to assuming that $T=Z$. Then $f^![T]=[D]$, and $\delta \cap [T]=[D]$ (this computation is carried out in \cite[\S2]{Gil-K-87}).

The map, where $n\in\mathbb{N}$ and $Z$ runs over the closed subschemes of $T$,
\[
\oplus_{n,Z} \HK{n}{Z}{T} \to \oplus_{n,Z} \B^n(Z \to T)
\]
is a ring morphism by \cite[Theorem~8.2]{Gil-Ri-81}. It follows from the anti-commutativity of Quillen $K$-theory that the ring on the left is commutative. Thus the elements of the image of this morphism commute with one another, and in particular $\alpha$ commutes with $\delta=[f]$, as required.
\end{example}

\begin{lemma}
\label{lemm:RR}
Let $S$ be a regular excellent scheme, and $X$ an integral closed subscheme of $S$. Assume that $n=-\dim_S X >0$. There is a finite resolution $\Ec_\bullet$ of $\Oc_X$ by locally free coherent $\Oc_S$-modules, and for any such complex $\Ec_\bullet$, we have $c_n(\Ec_\bullet) \cap [S] = (-1)^{n-1}(n-1)! \cdot [X]$ in $\CH_{-n}(X)$.
\end{lemma}
\begin{proof}
The resolution $\Ec_\bullet$ exists because $S$ is regular, and moreover the class $c_n(\Ec_\bullet)$ is independent of the choice of that resolution. The singular locus $\Sing X$ is closed and not dense in $X$, because $X$ is quasi-excellent and integral. It follows from the localisation sequence that the morphism $\CH_{-n}(X) \to \CH_{-n}(X - \Sing X)$ is an isomorphism. We may therefore replace $S$ and $X$ by the respective complements of $\Sing X$, and assume that $X$ is regular. Then $X \to S$ is a regular closed embedding \cite[(19.1.1)]{ega-4}; we let $N$ be its normal bundle. We consider the deformation to the normal cone construction \cite[\S 6.1]{Ful-In-98}
\[ \xymatrix{
P\ar[r]^{j_\infty} & W &S\ar[l]_{j_0}\\ 
X \ar[r]^{i_\infty} \ar[u] & \Aa^1_X\ar[u]&X\ar[l]_{i_0}\ar[u]
}\]
where $P=\Pp(N \oplus 1)$, and $W$ is the blow-up of $\Aa^1_S$ along $X$ via the composite $X \xrightarrow{i_0} \Aa^1_X \to \Aa^1_S$. Each of the schemes $P$ and $\Aa^1_X$, being smooth over $X$, is regular by \cite[(17.5.8, (iii))]{ega-4}. 
The scheme $W$ is regular ($W-P$ is regular, being an open subscheme of $X$, and $W$ is regular at every point of $P$ by \cite[(19.1.1)]{ega-4}), hence we can find a finite resolution $E_\bullet$ of $\Oc_{\mathbb{A}^1_X}$ by locally free coherent $\Oc_W$-modules. Then $j_\infty^*E_\bullet$, resp.\ $j_0^*E_\bullet$, is a finite resolution of $\Oc_X$ by locally free coherent $\Oc_P$-modules, resp. $\Oc_S$-modules. Using the transversality of the squares in the diagram above, and the relation $i_0^*=i_\infty^* \colon \CH_{-n}(\Aa^1_X) \to \CH_{-n}(X)$, we have in $\CH_{-n}(X)$
\[
c_n(j_0^*E_\bullet ) \cap [S] = i_0^*(c_n(E_\bullet)\cap [W]) = i_\infty^*(c_n(E_\bullet) \cap [W]) = c_n(j_\infty^*E_\bullet ) \cap [P].
\]
Thus we may replace $S$ by $P$, and assume that $f \colon X \to S$ has a section $p\colon S \to X$. Then we have in $\CH_{-n}(S)$ (see the proof of \cite[Proposition~18.1 (a)]{Ful-In-98})
\[
f_*(c_n(\Ec_\bullet)\cap [S]) = \sum_i (-1)^i c_n(\Ec_i)[S] = c_n(f_*[\Oc_X])[S].
\]
This cycle coincides with $(-1)^{n-1}(n-1)! \cdot [X]$ by \cite[Example~15.3.1, \S20]{Ful-In-98}, and we conclude by applying $p_* \colon \CH_{-n}(S) \to \CH_{-n}(X)$.
\end{proof}

\begin{theorem}
\label{th:poincare}
Let $S$ be a regular excellent scheme, and $X$ a closed subscheme of $S$. Then for $n\leq 2$, the association $\alpha \mapsto \alpha \cap [S]$ induces an isomorphism
\[
\A^n(X \to S) \to \CH_{-n}(X).
\]
\end{theorem}
\begin{proof}
Let $\alpha \in \A^n(X \to S)$ be such that $\alpha \cap [S]=0$. Let $r \colon R \to S$ be a quasi-projective morphism, with $R$ regular. By \autoref{lemm:decomp_qproj}, we can write $r = p\circ i$, with $i$ a regular closed embedding and $p$ a smooth morphism. We have in $\CH(R \times_S X)$
\[
\alpha \cap [R] =(\alpha \cdot [i] \cdot [p]) \cap [S] = ([i] \cdot [p] \cdot \alpha) \cap [S]=0.
\]
Let $d$ be an integer. For any scheme $Y$ of finite type over $S$, we let $\rho_Y \colon \CH_d(Y) \to \CH_{d-n}(Y \times_S X)$ be defined by $\rho_Y(x)=\alpha\cap x$. Then $\rho$ satisfies the conditions of \autoref{lemm:test1} for $T=S$, with $G$ as in \autoref{ex:GCH} ($C=\Zz$, $j=d$) and $H$ as in \autoref{ex:Hreldim} ($r=0$, $j=d-n$). This proves that $\alpha=0$, showing injectivity.

Let $i \colon Z \to X$ be a closed embedding, with $Z$ integral and $\dim_S Z = -n \leq 0$. There is a bounded complex $\Ec_\bullet$ of locally free coherent $\Oc_S$-modules, exact off $Z$, such that $c_n(\Ec_\bullet) \cap [S]=(-1)^{n-1}\cdot [Z] \in \CH_{-n}(Z)$; this follows from \autoref{lemm:RR} when $n \in \{1,2\}$, and we may take $\Ec_\bullet=\Oc_Z[-1]$ when $n=0$. Therefore $\alpha=i_*(c_n(\Ec_\bullet))$ is an antecedent of the class $(-1)^{n-1}\cdot[Z]$ under the map of the statement. Since such classes generate the group $\CH_{-n}(X)$, we have proved surjectivity.
\end{proof}

Let us record the following statement, obtained in the course of the proof above.
\begin{lemma}
\label{sch:cne}
Let $S$ be a regular excellent scheme, and $X$ a closed subscheme of $S$. For $n\leq 2$, the group $\A^n(X \to S)$ is generated by the elements $i_*(c_n(\Ec_\bullet))$, where $i\colon Z \to X$ is a closed embedding, and $\Ec_\bullet$ a bounded complex of locally free coherent $\Oc_S$-modules, which is exact off $Z$.
\end{lemma}

\begin{proposition}
\label{cor:comm}
Let $S$ be a regular excellent scheme, $X$ a closed subscheme of $S$, and $Y$ a scheme of finite type over $S$. Then for any $\alpha \in \A^n(X \to S)$, with $n \leq 2$, and $\beta \in \A(Y \to S)$, we have $\alpha \cdot \beta = \beta \cdot \alpha$.
\end{proposition}
\begin{proof}
We may assume that $\alpha=i_*(c_n(\Ec_\bullet))$, as in \autoref{sch:cne}. Since the construction of the class $c_n(\Ec_\bullet)$ only uses proper push-forwards, smooth pull-backs, and Gysin maps, one checks that $\beta \cdot c_n(\Ec_{\bullet})=c_n(\Ec_{\bullet}) \cdot \beta$. But $\beta$ is compatible with $i_*$, by definition of $\B^m(Y\to S)$.
\end{proof}

\begin{proposition}
\label{cor:commt}
Let $S$ be a regular excellent scheme, and $Y \to T \to S$ morphisms of finite type. Let $X$ be a closed  subscheme of $T$. Assume that $X \times_T Y \to T$ has relative dimension $\leq r$ (\autoref{ex:Hreldim}). For any $\alpha \in \A^n(X \to T)$ and $\beta \in \A^m(Y \to T)$, with $m+n +r \leq 2$ and $n\leq 2$, we have $\alpha \cdot \beta = \beta \cdot \alpha$.
\end{proposition}
\begin{proof}
Let $d$ be an integer. For any scheme $B$ of finite type over $T$, we let $\rho_B \colon \CH_d(B) \to \CH_{d-m-n}(B\times_T X \times_T Y)$ be defined by $\rho_B(x)=(\alpha \cdot \beta - \beta \cdot \alpha)\cap x$. Let $R$ be a regular scheme of finite type over $T$. We have $\rho_R[R]=0$ by \autoref{cor:comm} since $n\leq 2$. We conclude using \autoref{lemm:test1}, with $G$ as in \autoref{ex:GCH} ($C=\Zz$, $j=d$) and $H$ as in \autoref{ex:Hreldim} ($j=d-r-m-n$).
\end{proof}

\begin{remark}
\label{rem:commt}
\autoref{cor:commt} applies in particular in the following situations.
\begin{enumerate}[label=(\roman{*}), ref=(\roman{*})]
\item $Y$ is a closed subscheme of $T$, and $m=n=1$, $r=0$. This says that elements of degree one commute in the ring $\bigoplus_Z \A(Z \to T)$, where $Z$ runs over the closed subschemes of $T$.

\item \label{rem:flat} $f \colon Y \to T$ is flat of constant relative dimension $r=-m$, $\beta=[f]$, $n \leq 2$. This says that any element of $\A^n(X \to T)$, with $n\leq 2$, is automatically compatible with flat pull-backs. Therefore our definition of $\A^n(X \to T)$ agrees with that of \cite[\S17]{Ful-In-98} when $n\leq 2$ and $X \to T$ is a closed embedding.
\end{enumerate}
\end{remark}

\begin{remark}
When $S$ is a field of characteristic zero, then bivariant classes of arbitrary degrees commute \cite[Example~17.4.4]{Ful-In-98}. 

When $S$ smooth over a field, and $X$ a scheme of finite type over $S$, then the map $\A(X\to S) \to \CH(X)$ is known to be an isomorphism \cite[Propositions~17.3.1 and 17.4.2]{Ful-In-98}.
\end{remark}

The results of this section can be improved in characteristic zero, as follows.

\begin{theorem}
\label{th:Q}
Let $S$ be a regular excellent $\Qq$-scheme, and $X$ a scheme of finite type over $S$.
\begin{enumerate}[label=(\roman{*}), ref=(\roman{*})]
\item \label{Q:poincare} The morphism $\A(X \to S) \to \CH(X)$ is an isomorphism. 

\item \label{Q:com} If $A$ and $B$ are schemes of finite type over $X$, and $\alpha \in \A(A \to X)$, $\beta \in \A(B \to X)$, we have $\alpha \cdot \beta = \beta \cdot \alpha$.

\item \label{Q:dim} For any scheme $Y$ finite type over $X$, the group $\A(Y \to X)$ coincides with that of \cite[\S17,\S20]{Ful-In-98}.
\end{enumerate}
\end{theorem}
\begin{proof}
We obtain the injectivity in \ref{Q:poincare} by proceeding as in the proof of \autoref{th:poincare}, using \autoref{rem:Q}.

To prove surjectivity in \ref{Q:poincare}, it will be sufficient to assume that $X$ is integral, and find an element $\alpha \in \A(X \to S)$ such that $\alpha \cap [S]=[X]$.  By Chow's lemma \cite[(5.6.1)]{ega-2}, we may assume that $X$ is quasi-projective over $S$, so that by \autoref{lemm:decomp_qproj}, the morphism $X\to S$ decomposes as $p\circ i$,where $i\colon X \to W$ is a closed embedding and $p\colon W \to S$ is smooth. Then for any $\beta \in \A(X \to W)$, we have $\beta \cap [W] = (\beta \cdot [p]) \cap [S]$, so that we may replace $S$ by $W$ (the latter is regular by \cite[(17.5.8, (iii))]{ega-4}), and assume that $X \to S$ is a closed embedding. Since Gersten's conjecture is true for the local rings of $S$ by \cite{Panin-Gersten_conjecture}, the composite
\[
\HK{n}{X}{S} \to \A^n(X \to S) \to \CH_{-n}(X)
\]
is an isomorphism (this is Bloch's formula, see for instance \cite[\S8 and Corollary~7.20]{Gil-Ri-81}). This proves surjectivity in \ref{Q:poincare}. 

Let us mention that we actually proved that any element of $\A^n(X \to S)$ decomposes as $\pi_*(\alpha \cdot [p])$, where $\pi \colon X'\to X$ is a proper morphism, $p \colon W \to S$ is smooth, $X'$ is a closed subscheme of $W$, and $\alpha \in \HK{n+\dim_S W}{X'}{W}$. 

If $A$ and $B$ are two closed subschemes of $S$, and $\alpha \in \HK{n}{A}{S}$, $\beta \in \HK{m}{B}{S}$, we have  $\alpha \cdot \beta = \beta \cdot \alpha$ in $\A^{m+n}(A \cap B \to S)$ (see the last paragraph of \autoref{ex:chowcoh}). From the decomposition $\pi_*(\alpha\cdot [p])$ mentioned above, we obtain \ref{Q:com} under the additional assumption that $X$ is regular. Then we remove this assumption by reasoning as in the proof of \autoref{cor:commt} (using \autoref{rem:Q}).

Finally \ref{Q:dim} follows from \ref{Q:com}, as in \autoref{rem:commt} \ref{rem:flat}.
\end{proof}

\paragraph{\textbf{Acknowledgements.}} I am grateful to the referee for his remarks and explanations, which helped improve the exposition.

\bibliographystyle{alpha}

\begin{thebibliography}{Hau12b}
\bibitem[Bro03]{Bro-St-03}
Patrick Brosnan.
\newblock Steenrod operations in {C}how theory.
\newblock {\em Trans. Amer. Math. Soc.}, 355(5):1869--1903 (electronic), 2003.

\bibitem[Ful98]{Ful-In-98}
William Fulton.
\newblock {\em Intersection theory}, volume~2 of {\em Ergebnisse der Mathematik
  und ihrer Grenzgebiete. 3. Folge. A Series of Modern Surveys in Mathematics}.
\newblock Springer-Verlag, Berlin, second edition, 1998.

\bibitem[Gil81]{Gil-Ri-81}
Henri Gillet.
\newblock Riemann-{R}och theorems for higher algebraic {$K$}-theory.
\newblock {\em Adv. in Math.}, 40(3):203--289, 1981.

\bibitem[Gil87]{Gil-K-87}
Henri Gillet.
\newblock {$K$}-theory and intersection theory revisited.
\newblock {\em $K$-Theory}, 1(4):405--415, 1987.

\bibitem[EGA I]{ega-1}
Alexander Grothendieck.
\newblock \'{E}l\'ements de g\'eom\'etrie alg\'ebrique. {I}. {L}e langage des
  sch\'emas.
\newblock {\em Inst. Hautes \'Etudes Sci. Publ. Math.}, (4):228, 1960.

\bibitem[EGA II]{ega-2}
Alexander Grothendieck.
\newblock \'{E}l\'ements de g\'eom\'etrie alg\'ebrique. {II}. \'{E}tude globale
  \'el\'ementaire de quelques classes de morphismes.
\newblock {\em Inst. Hautes \'Etudes Sci. Publ. Math.}, (8):222, 1961.

\bibitem[EGA $\textsc{IV}$]{ega-4}
Alexander Grothendieck.
\newblock \'{E}l\'ements de g\'eom\'etrie alg\'ebrique. {IV}. \'{E}tude locale des sch\'emas et des morphismes de sch\'emas.
\newblock {\em Inst. Hautes \'Etudes Sci. Publ. Math.}, (20), 1964; ibid. (24), 1965; ibid. (28), 1966; ibid. (32), 1967.

\bibitem[Hau12a]{firstst}
Olivier Haution.
\newblock Integrality of the {C}hern character in small codimension.
\newblock {\em Adv. Math.}, 231(2):855--878, 2012.

\bibitem[Hau12b]{reduced}
Olivier Haution.
\newblock Reduced {S}teenrod operations and resolution of singularities.
\newblock {\em J. K-Theory}, 9(2):269--290, 2012.

\bibitem[Hau13]{firstsq}
Olivier Haution.
\newblock On the first {S}teenrod square for {C}how groups.
\newblock {\em Amer. J. Math.}, 135(1):53--63, 2013.

\bibitem[Hau13b]{duality}
Olivier Haution.
\newblock Duality and the topological filtration.
\newblock {\em Math. Ann.}, 357(4):1425--1454, 2013.

\bibitem[Hir64]{Hir-64}
Heisuke Hironaka.
\newblock Resolution of singularities of an algebraic variety over a field of
  characteristic zero. {I}, {II}.
\newblock {\em Ann. of Math. (2)}, 79:109--203, 205--326, 1964.

\bibitem[ILO]{Gabber-book}
Luc Illusie, Yves Laszlo, and Fabrice Orgogozo.
\newblock {\em Travaux de Gabber sur l'uniformisation locale et la cohomologie
  \'etale des sch\'emas quasi-excellents. S\'eminaire \`a l'Ecole polytechnique
  2006--2008}.
\newblock Avec la collaboration de Fr\'ed\'eric D\'eglise, Alban Moreau,
  Vincent Pilloni, Michel Raynaud, Jo\"el Riou, Beno\^it Stroh et Michael
  Temkin. \href{http://arxiv.org/abs/1207.3648}{\tt{arXiv:1207.3648}}.

\bibitem[Kar03]{Kar-On-2003}
Nikita~A. Karpenko.
\newblock On the first {W}itt index of quadratic forms.
\newblock {\em Invent. Math.}, 153(2):455--462, 2003.

\bibitem[Kim92]{Kimura-Fractional}
Shun-ichi Kimura.
\newblock Fractional intersection and bivariant theory.
\newblock {\em Comm. Algebra}, 20(1):285--302, 1992.

\bibitem[Lip78]{Lipman-desingularization}
Joseph Lipman.
\newblock Desingularization of two-dimensional schemes.
\newblock {\em Ann. Math. (2)}, 107(1):151--207, 1978.

\bibitem[Pan03]{Panin-Gersten_conjecture}
Ivan~A. Panin.
\newblock The equicharacteristic case of the {G}ersten conjecture.
\newblock {\em Tr. Mat. Inst. Steklova}, 241(Teor. Chisel, Algebra i Algebr.
  Geom.):169--178, 2003.

\bibitem[SGA 6]{sga6}
{\em Th\'eorie des intersections et th\'eor\`eme de {R}iemann-{R}och}.
\newblock Lecture Notes in Mathematics, Vol. 225. Springer-Verlag, Berlin,
  1971.
\newblock S{\'e}minaire de G{\'e}om{\'e}trie Alg{\'e}brique du Bois-Marie
  1966--1967 (SGA 6), Dirig{\'e} par P. Berthelot, A. Grothendieck et L.
  Illusie. Avec la collaboration de D. Ferrand, J. P. Jouanolou, O. Jussila, S.
  Kleiman, M. Raynaud et J. P. Serre.

\bibitem[Sou85]{Sou-Op-85}
Christophe Soul{\'e}.
\newblock Op\'erations en {$K$}-th\'eorie alg\'ebrique.
\newblock {\em Canad. J. Math.}, 37(3):488--550, 1985.

\bibitem[Tem08]{Temkin-desing}
Michael Temkin.
\newblock Desingularization of quasi-excellent schemes in characteristic zero.
\newblock {\em Adv. Math.}, 219(2):488--522, 2008.

\bibitem[Tem12]{Temkin-duke}
Michael Temkin.
\newblock Functorial desingularization of quasi-excellent schemes in
  characteristic zero: the nonembedded case.
\newblock {\em Duke Math. J.}, 161(11):2207--2254, 2012.

\bibitem[Voe03]{Vo-03}
Vladimir Voevodsky.
\newblock Reduced power operations in motivic cohomology.
\newblock {\em Publ. Math. Inst. Hautes \'Etudes Sci.}, (98):1--57, 2003.
\end{thebibliography}

\end{document}